\documentclass[12pt]{amsart}
\usepackage{graphicx}
\usepackage{verbatim}
\usepackage{textcomp}
\usepackage{amssymb}
\usepackage{cite}
\usepackage{color}
\usepackage[all]{xy}
\usepackage{graphicx}
\usepackage{color}
\usepackage{hyperref}
\hypersetup{
    colorlinks,
    citecolor=black,
    filecolor=black,
    linkcolor=black,
    urlcolor=black
}
\usepackage{tikz}
\usepackage[all]{xy}
\usepackage{graphicx}
\usetikzlibrary{backgrounds}
\newtheorem{theorem}{Theorem}[section]
\newtheorem{lemma}[theorem]{Lemma}

\newtheorem{prop}[theorem]{Proposition}

\allowdisplaybreaks

\theoremstyle{definition}
\newtheorem{definition}[theorem]{Definition}

\theoremstyle{remark}
\newtheorem{remark}[theorem]{Remark}

\numberwithin{equation}{section}

\newtheorem{thm}{Theorem}[section]

\theoremstyle{definition}

\newtheorem{defn}[thm]{Definition}

\newcommand{\R}{\mathbb R}

\newcommand{\C}{\mathbb C}

\newcommand{\B}{\mathcal B}

\date{\today}

\DeclareMathOperator{\re}{Re}			
\DeclareMathOperator{\im}{Im}			

\title[Bubble tree convergence for harmonic maps into metric spaces]{Bubble tree convergence for harmonic maps into compact locally CAT(1) spaces}

\thanks{CB was supported in part by NSF grant DMS-1609198.}
\author[Breiner]{Christine Breiner}
\address{Department of Mathematics \\
                 Fordham University \\
                 Bronx, NY  10458}
\email{cbreiner@fordham.edu}
\author[Lakzian]{Sajjad Lakzian}
\address{Department of Mathematics \\
                 Fordham University \\
                 Bronx, NY  10458}
                 \email{slakzian@fordham.edu}

\begin{document}
\maketitle
\begin{abstract}
We determine bubble tree convergence for a sequence of harmonic maps, with uniform energy bounds, from a compact Riemann surface into a compact locally CAT(1) space. In particular, we demonstrate energy quantization and the no-neck property for such a sequence. In the smooth setting, Jost \cite{jost} and Parker \cite{Parker} respectively established these results by exploiting now classical arguments for harmonic maps. Our work demonstrates that these results can be reinterpreted geometrically. In the absence of a PDE, we take advantage of the local convexity properties of the target space. Included in this paper are an $\epsilon$-regularity theorem, an energy gap theorem, and a removable singularity theorem for harmonic maps for harmonic maps into metric spaces with upper curvature bounds. We also prove an isoperimetric inequality for conformal harmonic maps with small image. 
\end{abstract}

In pioneering work, Sacks and Uhlenbeck \cite{SU} determined a priori estimates for critical points to a perturbed energy functional to prove the existence of minimal two-spheres in compact Riemannian manifolds. Recently, Breiner et al. \cite{Banff2} extended this result to the singular setting. Lacking a PDE, they instead used the local convexity of the target space (a locally CAT(1) space) to determine a discrete harmonic map heat flow or harmonic replacement process. Given a finite energy map $\phi:(M^2,g) \to (X,d)$, harmonic replacement yields either a harmonic map $u:(M^2,g) \to (X,d)$, homotopic to $\phi$, or a conformal harmonic map $v:(\mathbb S^2,g_0) \to (X,d)$.  
The second case occurs when the modulus of continuity for the sequence of replacement maps blows up at a point. Renormalizing the domain on the scale of the blow up gives a sequence of finite energy maps with uniform modulus of continuity, which converge in $C^0$ uniformly on compact sets to the map $v:\mathbb C \to X$. By proving a removable singularity theorem for \emph{conformal} harmonic maps, in \cite{Banff2} they concluded that $v$ is harmonic on $\mathbb S^2$. 

In the smooth setting, Sacks and Uhlenbeck \cite{SU} proved a removable singularity theorem for harmonic maps; coupling this result with their a priori estimates, any bounded sequence of critical maps has a subsequence $\{u_k\}$ which converges, away from some finite number of points, to a harmonic map $u$. At these points, renormalizing the domain by the norm of the gradient produced a ``bubble'', a harmonic map $\omega_i$ from $\mathbb S^2$. 
Jost \cite{jost} demonstrated that the energy in this process is quantized; that is, 
\[
\lim_{k \to \infty} E[u_k] = E[u] + \sum_{i=1}^\ell E[\omega_i].
\]
Parker \cite{Parker} (see also \cite{ParkerWolfson}) refined the renormalization technique of \cite{SU} to produce a full accounting of the $C^0$ limit picture for harmonic maps to smooth targets. Further, he provided a new proof of energy quantization along with the first proof of the so called ``no-neck property". 

We demonstrate that the energy quantization and no-neck property determined for harmonic maps to smooth manifolds is also valid for harmonic maps into metric spaces with upper curvature bounds. Unlike in the smooth setting, the best convergence one can hope for is $C^0$ uniform convergence. While the work of \cite{Parker} relied on $C^1$ convergence, we demonstrate that the weaker convergence does not impede energy quantization or the no-neck property.

\begin{theorem}\label{MAIN}
Let $u_k:(M^2,g) \to (X,d)$ be a sequence of finite energy harmonic maps from a compact Riemann surface $(M^2,g)$ to a compact locally CAT(1) space $(X,d)$. If $E[u_k] \leq \Lambda <\infty$ then there exists a subsequence again denoted by $\{u_k\}$ and a bubble tower domain $T$ such that the bubble tree maps $\overline {\underline u}_{k,I}:T \to X$ converge to a harmonic bubble tree map $\overline{\underline u}_I:T \to X$ in $C^0$ uniformly on $T$ and $E[\overline {\underline u}_{k,I},T] \to E[\overline {\underline u}_{I},T]$. In addition
\begin{enumerate}
\item energy is quantized: 
\[
\lim_{k \to \infty} E[u_k,M] = E[\overline {\underline u}_{I},T].
\]
\item the no-neck property holds: at each bubble point $y_{i_1\dots i_n}$, the image of the map $\overline {\underline u}_{i_1 \dots i_{n-1}}$ at vertex $i_1 \dots i_n$ and the image of the map $\overline{\underline u}_{i_1 \dots i_n}$ meet. That is $\overline{\underline u}_{i_1\dots i_{n-1}}(y_{i_1 \dots i_n}) = \overline {\underline u}_{i_1 \dots i_n}(p^+)$.
\end{enumerate}
\end{theorem}

\begin{remark}
The bubble tree domain, the maps $\overline{\underline u}_{k,I}$, and their limit maps are defined in Section \ref{BTconst}. Also, $p^+\in \mathbb S^2$ denotes the north pole.
\end{remark}

For those unfamiliar with the bubble tree construction, we provide some context. The essence of the theorem is that the sequence $u_k$ may develop energy concentration points in the limit. Away from the concentration points, the maps converge in $C^0$ uniformly to a harmonic map $u:M \to X$. By rescaling appropriately at each of the concentration points $x_i$, we produce sequences of maps $\overline u_{k,i}$ with domains exhausting $\mathbb S^2$ and which themselves converge, away from some finite number of concentration points, to a conformal harmonic map $\overline u_i:\mathbb S^2 \to X$. The process iterates at these new concentration points to produce bubbles on bubbles. The energy quantization and the no-neck property imply that all of the essential information is contained in the limit base map and bubbles. This follows since, on the annular domain between the scale on which the $u_k$'s converge and the scale on which the $\overline u_{k,i}$'s converge, the energy and the diameter both tend to zero in the limit.

The paper is organized as follows. In Section \ref{Prelims} we define the relevant terms and provide necessary background or references to it. Section \ref{Tools} contains statements and proofs of the four main tools needed to start a compactness theorem. These include an $\epsilon$-regularity theorem, a gap theorem, a convergence theorem, and a removable singularity theorem for harmonic maps. Notice that the removable singularity theorem extends \cite[Theorem 3.6]{Banff2}, which proved the result for \emph{conformal} harmonic maps. In Section \ref{IsoperimetricSection} we prove Theorem \ref{thm:isoperimetric}, an isoperimetric inequality for conformal harmonic maps into compact locally CAT(1) spaces with small area and small image.  In Section \ref{BTC}, we prove Theorem \ref{MAIN} using the tools from Sections \ref{Tools} and \ref{IsoperimetricSection} and following the general outline of \cite{Parker}. 

\section{Preliminaries}\label{Prelims}

Throughout the paper we let $(M, g)$ denote a compact Riemann surface with a smooth metric and let $(X, d)$ denote a compact locally CAT(1) space. We refer the reader to  \cite[Section 2.2]{Banff1} for background on CAT(1) spaces. A metric space $(X,d)$ is said to be \emph{locally} CAT(1) if every point of $X$ has a geodesically convex CAT(1) neighborhood. Note that for a compact locally CAT(1) space, there exists a radius $r(X)>0$ such that for all $P \in X$, $\overline {\B_{r(X)}(P)}$ is a compact CAT(1) space. Let
\[
\tau(X):= \min\{r(X), \pi/4\}
\]and let $\mathrm{inj}(M)$ denote the injectivity radius of $M$.

For $r \in (0, \mathrm{inj}(M))$, $t \in (0, \tau(X))$, we denote geodesic disks and balls in their respective domains as $D_r(x) \subset M$ and $\B_t(P) \subset X$. We also frequently consider geodesic disks with respect to the metric induced by the pullback of the exponential map and use the same notation, $D_r(0) \subset T_xM =\mathbb R^2$.

Following the definition in \cite{korevaar-schoen1}, the Sobolev space $W^{1,2}(M,X)$ is the space of finite energy maps. That is, $u \in W^{1,2}(M,X)$  if its energy density function (as defined in \cite{korevaar-schoen1}) $|\nabla u|^2 \in L^1(M)$. The total energy of the map $u$ is given by
\[
E[u]:=\int_M|\nabla u|^2 d\mu_g
\]and we denote the energy on subsets $\Omega \subset M$ by
\[
E[u,\Omega]:= \int_\Omega |\nabla u|^2 d\mu_g.
\]Given any $h \in W^{1,2}(\Omega,X)$ we define
\[
W^{1,2}_h(\Omega,X):=\{ f \in W^{1,2}(\Omega,X): Tr(f) = Tr(h)\}
\]where $Tr(u) \in L^2(\partial \Omega,X)$ denotes the trace map (see \cite{korevaar-schoen1}).  
\begin{defn}\label{def:min}
A map $u \in W^{1,2}(M,X)$ is \emph{harmonic} if it is locally energy minimizing. In particular, for each $x \in M$ there exist $0<r_x$, $0< \rho< \tau(X)$, and $P \in X$ such that $u(D_{r_x}(x)) \subset \B_\rho(P)$ and $h:=u|_{D_{r_x}(x)}$ has finite energy and minimizes energy among all maps in $W^{1,2}_h(D_{r_x}(x),\overline{ \B_\rho(P)})$. 
\end{defn} The existence and uniqueness of Dirichlet solutions follows from \cite[Lemma B.2]{Banff2} and \cite{serbinowski}. We will need also the regularity of such solutions.
\begin{theorem}[Lemma 1.3, \cite{Banff1}]
Suppose that $u:D_r \to \B_{\tau(X)}(P) \subset X$ is an energy minimizing map. Then $u$ is Lipschitz continuous on $D_{r/2}$ with Lipschitz constant depending only on $E[u, D_r]$ and $g$.
\end{theorem}

Let $|u_*(Z)|^2$ denote the directional energy density function for $Z \in \Gamma( TM)$, where $\Gamma(TM)$ is the space of Lipschitz vector fields on $M$ (see \cite[Section 1.8]{korevaar-schoen1}). For any finite energy map $u:(M,g) \to (X,d)$, let 
\[
\pi:\Gamma(TM) \times \Gamma(TM) \to L^1(M)
\]where 
\[
\pi(Z,W):= \frac 14 \left|u_*(Z+W)\right|^2 - \frac 14\left|u_*(Z-W)\right|^2.
\]By \cite[Lemma 3.5]{Banff1}, $\pi$ is a continuous, symmetric, bilinear, non-negative tensorial operator.

Let 
\begin{equation*}
	\Phi_u = \pi \left( {\partial_x},  {\partial_x}  \right)   -  \pi\left({\partial_y},{\partial_y}  \right) - 2 \mathbf{\textit{i}}\pi \left({\partial_x} , {\partial_y}\right)
\end{equation*}
denote the \emph{Hopf function} for $u$. As in the smooth setting, when $u$ is harmonic, $\Phi_u$ is holomorphic (see \cite[Lemma 3.7]{Banff2}).

\section{Analogues of Classical Results} \label{Tools}
In the smooth setting, the compactness follows from four properties of harmonic maps (see \cite[Proposition 1.1]{Parker}). We state an analogous proposition for harmonic maps into compact locally CAT(1) spaces. Note that the uniform convergence statement is not as strong as Parker's; we can get only $C^0$ uniform convergence. Nevertheless, we are still able to prove Theorem \ref{MAIN}. 

\begin{prop}\label{BTtools}
There exist positive constants $C', \epsilon'>0$ depending only on $(M, g)$ and $(X,d)$ such that the following hold:
\begin{enumerate}
\item (Sup Estimate) Let $u:D_r\to X$ be a harmonic map with $E\left[ u,D_r \right]< \epsilon'$ and $0<r<\epsilon'$. Then 
\[
\max_{0 \le \sigma \le r} \sigma^2 \sup_{D_{r - \sigma}} |\nabla u|^2 \le C'. 
\]In particular for all $x \in D_{3r/4}$, 
\[
 |\nabla u|^2(x) \leq \frac {C'}{r^2}.
\]	
\item (Energy Gap) If $(M,g)=(\mathbb S^2, g_0)$, where $g_0$ is the standard metric on the sphere, and $E\left[ u,\mathbb S^2 \right] < \epsilon'$, then $u$ is a constant map.  
\item (Uniform Convergence) Let $u_k:D_r \to X$ be a sequence of harmonic maps with $E\left[ u_k,D_r \right]<\epsilon'$. Then a subsequence $u_k$ convergence in $C^0$ uniformly to a harmonic map $u$ on $D_{r/2}$.
\item (Removable Singularity) Let $u:D_r \backslash \{0\} \to X$ be a finite energy harmonic map. Then $u$ extends to a locally Lipschitz harmonic map $u:D_r \to X$.

\end{enumerate}
\end{prop}
 The entirety of this section is devoted to proving each of these results. The results are listed in the order in which they are proven and each subsection contains the proof of a single item.

In the smooth setting, the proofs of these results rely on the Euler-Lagrange equation of the (perturbed) energy functional. Lacking such an equation, we instead exploit weak differential inequalities which follow from the locally minimizing property of harmonic maps coupled with the local convexity of the target space.

\subsection{Sup Estimate}

Following the now classical methods of \cite{Choi-Schoen}, we use a monotonicity formula and scale invariance to prove pointwise gradient bound for harmonic maps with small energy. 

\begin{prop}\label{GradientProp}
Suppose $u : D_r \to X$, $r \leq 1$, is a finite energy harmonic map. There exists an $\epsilon_0 >0$, depending only on the metric $g$, such that if $E\left[ u , D_r \right] < \epsilon_0$ and $r < \epsilon_0$, then 
\begin{equation}
	\max_{0 \le \sigma \le r} \sigma^2 \sup_{D_{r - \sigma}} |\nabla u|^2 \le C_0^2 ,
\end{equation} where $C_0$ depends only on the metric $g$. 
\end{prop}
Before proceeding with the proof, we point out an important subharmonicity estimate that we will need. The result follows from a local Bochner type inequality (see \cite{FZ}).

\begin{prop}\label{SubharmProp}
	Let $u : (D_{2r},g) \to X$ be a harmonic map with finite energy and let $g$ be a metric with bounded curvature. Then for all $\eta \in C_0^\infty(D_{r})$,  
	\begin{equation}
		-\int_{D_r} \nabla |\nabla u|^2 \cdot \nabla \eta \ge -C' \int_{D_r} \eta |\nabla u|^2 \left( 1+|\nabla u|^2\right)
	\end{equation}where $C'>0$ depends only on the curvature of the domain.
\end{prop}

\begin{proof}
For each $x \in \overline {D_r}$, let $s_x:= \sup\{s>0: u(D_s(x)) \subset \B_{\tau(X)}(u(x))\}$. For the open cover $\{D_{s_x}(x)\}_{x \in \overline{D_r}}$, consider a finite subcover $\left\{D_{s_i}(x_i)\right\}_{1\le i \le m}$ and denote $s:= \min_i\{s_{i}\}$. By \cite {FZ} there exists $C'>0$ depending on the curvature of $M$ such that for each $x_i$, 
\begin{equation}
-\int_{D_s(x_i)} \nabla  |\nabla u|^2 \cdot \nabla \eta \ge -C' \int_{D_s(x_i)} \eta |\nabla u|^2 \left( 1+|\nabla u|^2\right).
\end{equation}
Now let $\{\phi_i\}$ be a smooth partition of unity subordinate to the covering. Then $\phi_i \in C^\infty_0 (D_s(x_i))$ for each $i$. Moreover, $\sum_i \phi_i \equiv 1$, $\sum_i \nabla \phi_i \equiv 0$. Therefore for any test function $\eta \in C^1_0 (D_r)$, 
\begin{eqnarray*}
-\int_{D_r} \nabla |\nabla u|^2  \cdot \nabla \eta &=&- \int_{D_r} \nabla |\nabla u|^2  \cdot \nabla \left(\eta \left(\sum_i \phi _i\right)\right) \\ &=& -\sum_i \int_{D_s(x_i)\cap D_r} \nabla |\nabla u|^2 \cdot \nabla (\eta \phi_i ) \\ &\ge& -C' \sum_i  \int_{D_s(x_i)\cap D_r} \eta \phi_i |\nabla u|^2 \left( 1+|\nabla u|^2\right) \\
&=& -C' \int_{D_r} \; \eta|\nabla u|^2 \left( 1+|\nabla u|^2\right).
\end{eqnarray*}
	
\end{proof}

\begin{proof}[Proof of Proposition \ref{GradientProp}]
Choose $\sigma_0 \in (0 , r] $ and $x_0 \in \overline D_{r - \sigma_0}$ so that
\begin{equation*}
	 \sigma_0^2  \sup_{D_{r - \sigma_0}} |\nabla u|^2 = \max_{\sigma \in (0,r]} \sigma^2 \sup_{D_{r - \sigma}}  |\nabla u|^2,
\end{equation*} 
and
\begin{equation*}
	 |\nabla u|^2 (x_0) \geq \frac 12 \sup_{D_{r - \sigma_0}}  |\nabla u|^2.
\end{equation*}
We deduce that
\begin{equation*}
	\sup_{D_{\frac{\sigma_0}{2}}(x_0)} |\nabla u|^2 \le 8 |\nabla u|^2 (x_0).
\end{equation*}
Notice that if $\sigma_0^2 |\nabla u|^2 (x_0) \le 4$ then the desired result holds. So suppose instead that $ |\nabla u|^2 (x_0)  \ge 4 \sigma_0^{-2}$. Let $\tilde{u}: D_1 \to X $ be given by
\begin{equation*}
	\tilde{u}(x) = u \left( x_0 +|\nabla u|^{-1} (x_0)  x  \right)
\end{equation*} 
Then 
\begin{equation*}
	\sup_{D_1} \left|   \nabla \tilde{u} \right|^2 \le 8     \quad and \quad  \left|   \nabla \tilde{u} \right|^2(0) = 1.
\end{equation*}

By Lemma \ref{SubharmProp}, for all $\eta \in C_0^\infty(D_1)$, 
\begin{equation*}
	-\int_{D_1} \nabla \left|   \nabla \tilde{u}  \right|^2 \cdot \nabla \eta \ge -C'\int_{D_1}\eta |\nabla \tilde u|^2 \left( 1+|\nabla \tilde u|^2\right) \geq -9C' \int_{D_1}\eta \left|\nabla \tilde u\right|^2.
\end{equation*} Finally, Morrey's mean value inequality and the scale invariance of the energy implies that, for $c$ depending only on the domain metric $g$, 
\begin{equation*}
1 =	\left|   \nabla \tilde{u}  \right|^2 (0) \le   c  \int_{D_1}	\left| \nabla  \tilde{u}  \right|^2 \leq c \epsilon. 
\end{equation*}For $\epsilon$ sufficiently small, we get a contradiction. 
\end{proof}

\subsection{Energy Gap}

\begin{prop}[Energy Gap]\label{GapThm}There exists $\epsilon_{\mathrm{gap}}>0$ depending only on $g_0, (X,d)$ (where $g_0$ is the standard metric on $\mathbb S^2$) such that the following holds:

Let $u:\mathbb S^2 \to X$ be a conformal, harmonic map such that $E\left[ u,\mathbb S^2 \right] < \epsilon_{\mathrm{gap}}$. Then $u$ is a constant map.  
\end{prop}
\begin{proof}
Suppose first that $u(\mathbb S^2) \subset \mathcal B_{\tau(X)} (P)$ for some $P \in X$. Then, by \cite[Lemma 4.3]{Banff1}, $\Delta d^2(u(x),P) \geq \frac 12 |\nabla u|^2 \geq 0$ holds weakly on all of $\mathbb S^2$. It follows that $d^2(u(x),P) \equiv 0$, i.e. $u$ is constant. 

Now suppose that $u(\mathbb S^2)$ is not contained in $\mathcal B_{\tau(X)}(P)$ for all $P \in X$. Then, $\mathrm{diam}(u(\mathbb S^2)) > \tau(X)$. By the Monotonicity Formula of \cite[Theorem 3.4]{Banff2}, there exists $C>0$, independent of $u$ and $p \in \mathbb S^2$ such that 
\[
E\left[ u,\mathbb S^2 \right] \geq E\left[ u, u^{-1}(\mathcal B_{\tau(X)}(u(p))) \right] \geq C\tau(X)^2.
\] 

Thus, choosing $\epsilon_{\mathrm{gap}} <C \tau(X)^2$ implies the result.
\end{proof}

\subsection{Uniform Convergence}  
\begin{prop}There exists $\epsilon_2>0$, depending only on $g, (X,d)$ such that the following holds:

Let $u_k:D_r \to X$ be a sequence of harmonic maps with $E\left[ u_k,D_r \right]<\epsilon_2$. Then a subsequence $u_k$ converges in $C^0$ uniformly to a harmonic map $u$ on $D_{r/2}$.
\end{prop}
\begin{proof}Let $\epsilon_0, C_0$ be as in Proposition \ref{GradientProp}. Set $0<\epsilon_2  \leq \epsilon_0$. Then for all $x,y \in D_{3r/4}$ with $d_g(x,y)<\frac{\tau(X)}{C_0}r$ and all $k$, $d(u_k(x), u_k(y)) < \tau(X)$. Set $s:= \min\left\{\frac{\tau(X)}{C_0}r, \frac r{16}\right\}$. Cover $D_{r/2}$ by disks $\{D_{s/4}(x_j)\}$. By  \cite[Remark 3.2]{Banff2}, for all $k$, $u_k|_{D_{s}(x_i)}$ is energy minimizing. By \cite[Theorem 1.3]{Banff1}, the $u_k$ are equicontinuous on the cover $\{D_{s/2}(x_j)\}$. Therefore a subsequence $u_k \to u$ uniformly on every ball in this cover and thus on $D_{r/2}$. Applying \cite[Theorem 2.3]{Banff2} to each disk $D_{s/2}(x_j)$, we see that $u$ is energy minimizing on each disk $D_{s/4}(x_j)$. It follows that $u$ is harmonic on $D_{r/2}$.
\end{proof}

\subsection{Removable singularity theorem}
Notice that the work of this subsection extends the result of \cite[Theorem 3.6]{Banff2}, where a removable singularity theorem is proven for conformal harmonic maps. 
\begin{theorem}\label{RSThm}
Let $u:D_1 \backslash \{0\} \to X$ be a finite energy harmonic map. Then $u$ extends to a locally Lipschitz harmonic map $u:D_1 \to X$.
\end{theorem}

\begin{proof}
Since $u$ has finite energy, the Hopf function $\Phi_u\in L^1(D_1 \backslash \{0\}, \mathbb C \backslash \{0\})$ and therefore $\Phi_u$ can have at worst a simple pole at the origin. Without loss of generality, assume that $\Phi_u$ is nowhere zero on $D_1\backslash \{0\}$. 
 
We now follow the ideas of Schoen \cite[Theorem 10.4]{Schoen-Analytic} to define a conformal harmonic map. Schoen's argument involves taking the square root of $-\Phi_u$, and since in his case the domain is a disk and the image does not contain the origin the square root function is well-defined. We build an admissible cell complex $W$ (see \cite[Section 2.1]{Banff1}, \cite[Section 2.2]{daskal-meseCAG}) such that $W \setminus W^{(0)}$ will be the double cover of $\mathbb C \setminus \{0\}$. We then lift the map $\Phi_u$ to be defined from this double cover, allowing us to take its square root.

Let $H_j:=\{z \in \mathbb C: \im(z) \geq 0\}$, $j=1, \dots, 4$ denote four $2$-cells and let $z_j = x_j+ i y_j$ denote the coordinates in the $2$-cell $H_j$. Define the $2$-complex $W:= \bigsqcup_{j=1}^4H_j / \sim$ where the similarity relations determine the gluing of $1$-cell boundaries and are non-empty relations only in the following cases:
\[
\left\{ \begin{array}{ll}
z_1 \sim z_2 &\text{ iff } \re(z_1)=\re(z_2) \leq 0, \im(z_1)=\im(z_2)=0,\\
z_2\sim z_3 &\text{ iff } \re(z_2)=\re(z_3) \geq 0, \im(z_2)=\im(z_3)=0,\\
z_3 \sim z_4 &\text{ iff } \re(z_3)=\re(z_4) \leq 0,\im(z_3)=\im(z_4)=0,\\
z_4 \sim z_1 &\text{ iff } \re(z_4)=\re(z_1) \geq 0,\im(z_4)=\im(z_1)=0.
\end{array}\right.
\]It is straightforward to see that $W\backslash W^{(0)}$ is a double cover of $\mathbb C\backslash \{0\}$. We will associate each $p \in W$ with a projection onto $\mathbb C$ using isometries of half-spaces.

Let $\psi_j:H_j \to \{z \in \mathbb C: \im(z) \geq 0\}$, $ \psi_j^-:H_j \to \{z \in \mathbb C: \im(\overline z) \geq 0\}$ denote the natural Euclidean isometries. For $p \in W$, we define $\re, \im:W \to \R$ such that 
\[
\re(p):=
\re(\psi_j(z_j)) \text{ if } p=z_j, j=1,\dots, 4,
\]
\[
\im(p):=\left\{ \begin{array}{ll}
\im(\psi_j(z_j)) &\text{ if } p=z_j, j=1,3,\\
\im(\psi_j^-(z_j)) & \text{ if } p = z_j, j=2,4.
\end{array}\right.
\]Let $\Pi: W \to \mathbb C$ such that 
$\Pi(p):= \re(p) + i \im(p)$.
We define $\underline u: W \backslash W^{(0)} \to X$ and $\underline \Phi_u: W \backslash W^{(0)} \to \mathbb C \backslash \{0\}$ such that 
\[
\underline u(p) := u \circ \Pi(p); \quad \quad \underline \Phi_u(p):= \Phi_u \circ \Pi (p).
\]Note that $(\Phi_u)_*(\pi_1(D_1 \backslash \{0\})) = n \mathbb Z$ for some $n \in \mathbb N$. It follows that $(\underline \Phi_u)_*(\pi_1(W \backslash W^{(0)}))=2n\mathbb Z \subset 2\mathbb Z$. Therefore, there exists a map $\Psi_u:W \backslash W^{(0)}\to \mathbb C \backslash \{0\}$ such that $\Psi_u^2(p) = \underline \Phi_u(p)$. 
Define $v:W \backslash W^{(0)} \to \R$ such that
\[
v(p):= \re \int_{p_0}^p\Psi_u(\zeta)    d\zeta
\]where $p_0 \in W \backslash W^{(0)}$. By construction, $v$ is a well-defined, real-valued harmonic function which is minimizing on every compact subset of $W \backslash W^{(0)}$. 
We compute
	\begin{equation*}
	\frac{\partial v}{ \partial z}(p) = \frac{1}{2} \re \Psi(p) \in \R. 
	\end{equation*} It follows that $E[v] \leq C\int_{D_1\backslash \{0\}} |\Phi_u| d\mu_g< \infty$ and thus $v$ has finite energy. Let  $\tilde{u} :W \backslash W^{(0)}\to X \times \R$ where
\[
 \tilde{u}(p) := \left(\underline u(p)  ,  v(p)    \right).
\]  By definition, $\tilde u$ is a finite energy harmonic map and
the Hopf differential of $\tilde{u}$ satisfies
	\begin{equation*}
	{\Phi_{\tilde u}} (p) = \Phi_u(p) + 4 \left(\frac{\partial v}{\partial z} \right)^2 (p) \equiv 0.
	\end{equation*}
Therefore, $\tilde{u} :W \backslash W^{(0)} \to X \times \R$ is a \emph{conformal} harmonic map. We apply \cite[Theorem 3.6]{Banff2} to prove the removable singularity result for $\tilde u$. Observe that the hypothesis of the cited theorem states that the target space is \emph{compact} locally CAT(1) and that the domain is a Riemann surface. Nevertheless, the theorem can still be applied. 

While the metric space $(X \times \R, d \times \delta)$ is obviously not compact, it remains a locally CAT(1) space. Moreover, for each $P \in X \times \R$, the closed geodesic ball $\overline{\mathcal B_{\tau(X)}(P)}\subset X \times \R$ is a compact locally CAT(1) space. It follows that for any $\rho \in (0, \tau(X))$ and any $y \in X \times \R$, $\tilde u$ is energy minimizing on the domain $\tilde u^{-1}(\mathcal B_{\rho}(P))$. The removable singularity theorem for conformal harmonic maps does not in fact require compactness of the target space but does require a uniform lower bound in the target for which harmonic maps are minimizers. (This uniform lower bound is needed in order to appeal to the monotonicity formula.) Our target possesses such a uniform lower bound.

Moreover, while our domain here is a cell complex, away from $W^{(0)}$ the complex is a Riemannian manifold with a smooth Riemannian metric.
Therefore everywhere we apply the arguments of  \cite[Theorems 3.4 and 3.6]{Banff2} the fact that the domain is a complex is irrelevant. It follows that $\tilde u$ extends as a locally Lipschitz harmonic map $\tilde u:W \to X \times \R$ and thus so does $u$.
\end{proof}
\section{Isoperimetric Inequality for Minimal Surfaces with Small Area} \label{IsoperimetricSection}
We prove an isoperimetric inequality for minimal surfaces with small area in a CAT(1) metric space. By a minimal surface we mean a conformal harmonic map $u : \left( \Sigma , g \right) \to X$ which is minimizing in the sense of Definition \ref{def:min}. For such a map $u$, we define the area of its image by integrating the conformal factor $\lambda= \frac 12|\nabla u|^2$:
\begin{equation*}
	\mathrm{Area} \left( \mathrm{image}(u) \right) = \int_{\Sigma}\; \lambda \; d\mu_g.
\end{equation*} 
To prove the isoperimetric inequality we follow the classical arguments of Hoffman-Spruck \cite{HoffmanSpruck} who prove the result by first proving a Sobolev inequality for $C^1$ functions. 

We begin by improving the weak differential inequality satisfied by $d^2(u(x),Q)$ for some fixed $Q\in X$. 
\begin{lemma}\label{lem:2.5-modified} 
Given a geodesic triangle $\triangle PQS\subset X$ and $0 \leq \eta, \eta' \leq 1$, let $P_{\eta'}:= (1-\eta')P + \eta'Q$ and $S_\eta:=(1-\eta)S + \eta Q$. Then\begin{align*}
d^2 \left( P_{\eta'} ,  S_\eta \right) &\le \left( 1 - 2  \eta d_{QS} \cot d_{QS} \right) d^2_{PS} - 2 \left( \eta - \eta' \right) \left( d_{QS} - d_{QP} \right) d_{QS} + (\eta' - \eta)^2 d_{QS}^2  \\ 
		& \quad+\mathrm{Quad}(\eta, \eta') \mathrm{Quad}(d_{PS}, d_{QS}-d_{QP})+ \mathrm{Cub}\left( d_{PS}, d_{QS}-d_{QP}, \eta-\eta' \right).
\end{align*}
\end{lemma}
\begin{proof}
The proof follows from \cite[Lemmas 2.4 and 2.5]{Banff1} by keeping and expanding the equality 
\[
 \frac{\sin^2((1-\eta) d_{QS})}{\sin^2 d_{QS}} = \left(1 - \eta \frac {d_{QS}}{\sin d_{QS}} \cos d_{QS} +O(\eta^2)\right)^2
\]rather than getting an upper estimate.
\end{proof}
We now prove a modification of \cite[Lemma 4.3]{Banff1}, which implies almost subharmonicity for $d(Q,u(x))$. 
\begin{lemma}\label{lem:div-thm}
	Let $0 < t < \tau(X)$ and $u : \left( D_r , g   \right) \to \mathcal{B}_t (P) \subset X$ be an energy minimizing map. For a fixed $Q \in \mathcal{B}_t (P) $, $\eta \in [0,1]$, and all $0 < \sigma \le r$, 
	\begin{equation*}
			\int_{D_\sigma} \;   2 \eta \; \hat{d} \; \cot \hat{d} \;  \left| \nabla u  \right|^2  d\mu_g \le 	-\int_{D_\sigma} \; \left<  \nabla \eta , \nabla \hat{d}^2 \right> d\mu_g ;
	\end{equation*}
where $\hat{d} (x) := d \left(  Q , u(x) \right)$.
\end{lemma}
\begin{proof}
	Define 
$u_{\eta}:(D_\sigma,g) \rightarrow X$ 
by setting
\[
u_{\eta}(x)=(1-\eta(x)) u(x)+\eta(x) Q
\]
for $\eta \in C^{\infty}_c (D_\sigma)$. 
Letting $S = u(x),P = u(y),  \eta' =\eta(y)$, we use the estimate of Lemma \ref{lem:2.5-modified} to observe that  for $\hat d(x):= d(Q, u(x))$, 
\begin{align*}
d^2(u_\eta(y), u_\eta(x)) &\leq (1- 2\eta(x)\hat d(x)\cot(\hat d(x)))d^2(u(x),u(y)) \\& \quad
-2(\eta(x)-\eta(y))(\hat d(x)-\hat d(y))\hat d(x)\\
& \quad +(\eta(y)-\eta(x))^2\hat d^2(x) + \eta^2(x)\mathrm{Quad}(d(u(x),u(y)), \hat d(x)-\hat d(y)) \\
&\quad + \mathrm{Cub}\left( d(u(x),u(y)), \hat d(x)-\hat d(y), \eta(x)-\eta(y) \right).
\end{align*} The rest of the proof is identical to the rest of the proof of \cite[Lemma 4.3]{Banff1}.  
\end{proof}

\begin{lemma}
	Let $u: \Sigma\to X$ be a conformal harmonic map. Suppose $\xi \in C^1 \left( -\infty , \infty  \right)$ is a non-decreasing function such that $\xi(t) = 0$ for $t\le 0$, $h \in C_0^1(\Sigma)$ is a non-negative function, and $\xi h \in [0,1]$. For $x_0 \in \Sigma$ and $0<\rho<\tau(X)$, define
	\begin{equation*}
		\phi_{x_0} (\rho) := \int_\Sigma \;  h(x)  \; \xi (\rho - r(x)) \; \lambda(x) d\mu_g;
	\end{equation*}
and
	\begin{equation*}
	\psi_{x_0}(\rho) := \int_\Sigma  \;  \left| \nabla h \right|(x) \; \xi (\rho - r(x)) \; \lambda^{\frac{1}{2}}(x) d\mu_g
 	\end{equation*}	
 	where $r(x) := d\left( u(x) , u(x_0)  \right)$. Then the following differential inequality holds weakly:
\begin{equation}\label{eq:iso-1}
		- \frac{d}{d \rho} \left( \frac{\phi_{x_0}(\rho)}{\sin^2 \rho} \right) \le \frac{\psi_{x_0}(\rho)}{\sin^2\rho}.
\end{equation}
\end{lemma}
\begin{proof}
First note that \eqref{eq:iso-1} is equivalent to 
\begin{equation}\label{form2}
2 \cot \rho\, \phi_{x_0}(\rho) \le  \psi_{x_0}(\rho) +  \phi'_{x_0}(\rho).
\end{equation} By Lemma~\ref{lem:div-thm}, for any test function $\Psi\in [0,1]$ and $x_0 \in \Sigma$, we have that
\begin{equation*}
\int_{\Omega} \;    2 \Psi r \cot r \left| \nabla u  \right|^2  d\mu_g \le -	\int_{\Omega} \; \left<  \nabla \Psi , \nabla r^2 \right> d\mu_g 
\end{equation*}
where $\Omega:= u^{-1} \left( \B_\rho (u(x_0))      \right) $. 
Let $\Psi (x) = h(x) \xi \left( \rho - r(x) \right)$ so that 
\begin{equation*}
	\nabla \Psi (x) = -h(x) \xi'\left( \rho - r(x)  \right) \nabla r(x) + \xi \left( \rho - r(x)  \right) \nabla h(x).
\end{equation*}
By conformality and given the support of $\xi, \xi'$, it follows that
\begin{eqnarray}
2 \rho \cot \rho \; \int_{\Sigma} \;  \Psi \; \lambda d\mu_g  &\le& \int_{\Sigma} \;   \Psi r \cot r   \; |\nabla u|^2 d\mu_g \notag \\ &\le&    \int_{\Sigma} \; r(x) h(x) \xi'\left( \rho - r(x)  \right) \left|\nabla r(x) \right|^2 \; d\mu_g \notag \\ && -  \int_{\Sigma} \; r(x) \xi \left( \rho - r(x)  \right)\langle \nabla h(x) , \nabla r(x) \rangle\; d\mu_g \notag \\ &\le&    \notag \int_{\Sigma} \; r(x) h(x) \xi'\left( \rho - r(x)  \right) \; \lambda d\mu_g  \notag \\ && + \int_{\Sigma} \; r(x) \xi \left( \rho - r(x)  \right) \left| \nabla h(x) \right|  \lambda^{\frac{1}{2}} \; d\mu_g \notag \\ &\le& \rho \int_{\Sigma} \; h(x) \xi'\left( \rho - r(x)  \right) \; \lambda d\mu_g \notag \\ && + \rho \int_{\Sigma} \;  \xi \left( \rho - r(x)  \right) \left| \nabla h(x) \right|  \lambda^{\frac{1}{2}} \; d\mu_g. \notag 
\end{eqnarray}
Note that the string of inequalities implies that \eqref{form2} holds weakly.
\end{proof}
\begin{lemma}
	Let $u: \left(\Sigma , g \right) \to X$ be a minimal surface. Let $x_0 \in \Sigma$ with $h(x_0) \ge 1$. Let $\alpha$ and $t$ satisfy $0< \alpha < 1 \le t$. Set 
\begin{equation*}
	\rho_0 := \sin^{-1} \left( \frac{\int_\Sigma h(x) \; \lambda (x) \; d \mu_g}{\pi (1 - \alpha)}    \right)^{\frac 12} , 
\end{equation*} 
\begin{equation*}
\overline{\phi}_{x_0} (\rho) := \int_{S_\rho (x_0)} h(x) \lambda(x) d\mu_g , 
\end{equation*}
and 
\begin{equation*}
\overline{\psi}_{x_0}(\rho) := \int_{S_\rho (x_0)} \left| \nabla h(x) \right| \lambda^{\frac{1}{2}}(x) d\mu_g
\end{equation*}		where
\[
S_\rho(x_0):=\{ x \in \Sigma: d(u(x), u(x_0))< \rho\}.
\]
Then there exist $\rho$ with $0 < \rho < \rho_0$ such that 
\begin{equation*}
	\overline{\phi}_{x_0} (t \rho)  \le \alpha^{-1} \rho_0 \overline{\psi}_{x_0}(\rho);
\end{equation*}	
provided that 
\begin{equation*}
	\frac{\int_\Sigma h(x) \; \lambda (x) \; d \mu_g}{\pi (1 - \alpha)} \le 1
\end{equation*}	
and 
\begin{equation*}
	t \rho_0 \le \tau(X). 
\end{equation*}	
\end{lemma}
\begin{proof}
	The proof follows exactly the outline of \cite[Lemma 4.2] {HoffmanSpruck}, taking advantage of the differential inequality in \eqref{eq:iso-1} to establish a contradiction. 
\end{proof}
An argument similar to the covering argument used in \cite[Theorem 2.1]{HoffmanSpruck} (see also \cite{MichaelSimon}) immediately implies the following lemma. 
\begin{lemma}\label{lem:sob-type}
	Let $u: \left( \Sigma,g  \right) \to X$ be a conformal harmonic map with $\mathrm{image}(u) \subset \B_{\tau(X)} (P)$. If $ \mathrm{Area}\left[ u \left(   \Sigma  \right)  \right] \le \frac{\pi}{3} $,
then for any $h \in C^1(\Sigma)$,  
	\begin{equation*}
	\left( \int_\Sigma \; h^2(x) \;  \lambda(x) \; d \mu_g \right)^{\frac{1}{2}} \le \left( \frac{27 \pi}{4} \right)^{\frac{1}{2}} \int_\Sigma \; \left| \nabla h   \right|(x) \; \lambda^{\frac{1}{2}}\; d \mu_g. 
	\end{equation*}
\end{lemma}
Using the Sobolev type inequality of Lemma \ref{lem:sob-type} and an argument adapted from \cite{mese-iso}, we prove the isoperimetric inequality.
\begin{theorem}\label{thm:isoperimetric}
	Let $u: \left( \Sigma ,g  \right) \to X$ be a conformal harmonic map with $\mathrm{image}(u) \subset \B_{\tau(X)} (P)$. If $ \mathrm{Area}\left[ u \left(   \Sigma  \right)  \right] \le \frac{\pi}{3} $,
 then 
\begin{equation*}
	\frac{1}{2} E(u) = \mathrm{Area} \left[ u\left( \Sigma  \right) \right] \le \left( \frac{27 \pi}{4} \right) \mathrm{length}^2 \left[ u\left( \partial \Sigma \right) \right],
\end{equation*} 
\end{theorem}
\begin{proof}
	Since $u$ is uniformly continuous, for any $\epsilon>0$, we can pick a family of Lipschitz closed curves $\Gamma_\epsilon$ that approximate $\partial \Sigma$ i.e. with
	\begin{equation*}
		\left|   \mathrm{length} \left[   u(\Gamma_\epsilon)  \right] -  \mathrm{length} \left[ u\left( \partial \Sigma  \right) \right] \right| < \left(\frac 4{27\pi}\right)^{\frac 12} \epsilon. 
	\end{equation*}
	and such that
	\begin{equation*}
		\mathrm{Area}^{\frac 12}\left[ u\left( \Sigma \right)  \right] < \epsilon + \mathrm{Area}^{\frac 12}\left[ u\left( \Sigma_\epsilon  \right)  \right] 
	\end{equation*}
	where $\Sigma_\epsilon$ is the connected component of $\Sigma \backslash \Gamma_\epsilon$ which is disjoint from $\partial \Sigma$. By \cite[(1.9xvi)]{korevaar-schoen1}, for any Lipschitz closed curve $\Gamma \subset \Sigma$,
	\[
	\mathrm{length}[u(\Gamma)] = \int_\Gamma \lambda^{\frac 12} d\sigma_\Gamma.
	\]
Following the proof of \cite[Theorem 6.1]{meseCAG}, let $\lambda^\sigma:=e^{(\log \lambda)_\sigma}$, where $(\log \lambda)_\sigma$ is a symmetric mollification of $\log \lambda$. Then, $\lambda^\sigma \ge \lambda$ and $\sqrt{\lambda^\sigma} \to \sqrt{\lambda}$ in $W^{1,2}_{loc}(\Sigma)$. Then for any $h \in C^\infty_0(\Sigma)$ with $\|h\|_{L^\infty}<1$ and $\sigma>0$ sufficiently small,
\[
\int_\Sigma |\nabla h| \lambda^{\frac 12} d\mu_g \leq \int_\Sigma |\nabla h| (\lambda^\sigma)^{\frac 12}d\mu_g, 
\]
\[
\int_\Sigma h^2 \lambda^\sigma d\mu_g \leq \int_\Sigma h^2 \lambda d\mu_g +  \int_\Sigma(\lambda^\sigma - \lambda) d\mu_g.
\]By Lemma \ref{lem:sob-type}, 
\[
\left(\int_\Sigma h^2 \lambda^\sigma d\mu_g \right)^{\frac 12} \leq \left(\frac{27\pi}4\right)^{\frac 12} \int_\Sigma |\nabla h|(\lambda^\sigma)^{\frac 12}d\mu_g + O(\sigma).
\]
Using smooth approximations of the cutoff function on $\Sigma_\epsilon$, we observe that
\[
\left(\int_{\Sigma_\epsilon} \lambda^\sigma d\mu_g \right)^{\frac 12} \leq \left(\frac{27\pi}4\right)^{\frac 12} \int_{\Gamma_\epsilon}(\lambda^\sigma)^{\frac 12}d\sigma_{\Gamma_\epsilon} + O(\sigma), 
\]and letting $\sigma \to 0$ we see that
\[
\left(\int_{\Sigma_\epsilon} \lambda d\mu_g \right)^{\frac 12} \leq \left(\frac{27\pi}4\right)^{\frac 12} \int_{\Gamma_\epsilon}\lambda^{\frac 12}d\sigma_{\Gamma_\epsilon}.
\]By the choice of $\Gamma_\epsilon$, 
\[
	\mathrm{Area}^{\frac 12}\left[ u\left( \Sigma \right)  \right]  \leq \left( \frac{27 \pi}{4} \right)^{\frac{1}{2}} \int_{\partial \Sigma} \lambda^{\frac{1}{2}} d\sigma_{\partial \Sigma} + 2 \epsilon,
\]which implies the result.

\end{proof}

\section{Proof of the Main Theorem}\label{BTC}
This section consists of three subsections. In Section \ref{ConvSec}, we prove convergence results that produce the limit map and are applied iteratively to produce the bubble maps. Section \ref{MapsSec} contains a description of the bubble tree and the bubble maps. Finally, in Section \ref{NecSec} we prove the no-neck property and energy quantization result, which finishes the proof of Theorem \ref{MAIN}.

\subsection{Convergence Results}\label{ConvSec}
\begin{lemma}\label{BT1}
Let $u_k:(M,g) \to (X,d)$ be a sequence of harmonic maps such that $E\left[ u_k,M \right]<\Lambda<\infty$ and let $\epsilon_{\mathrm{gap}}$ be as in Proposition \ref{GapThm}. Then there exists a subsequence $\{u_k\}$ and a set of points $\{x_1, \dots, x_\ell\}$ with corresponding masses $\{m_1, \dots, m_\ell\}$ where $\ell \leq \Lambda/\epsilon_{\mathrm{gap}}$, and a harmonic map $u:M \to X$ such that
\begin{enumerate}
\item \label{bct1}$u_k \to u$ in $C^0$ uniformly on compact sets in $M \backslash \{x_1, \dots, x_\ell\}$.
\item \label{bct2}For any open subset $\Omega$ with $\overline{\Omega} \subset M \backslash \{x_1, \dots, x_\ell\}$, 
\[
\lim_{k \to \infty}E\left[ u_k, \Omega \right]=E\left[ u,\Omega \right].
\]
\item \label{bct3}  For all $r>0$ and all $i \in \{1, \dots, \ell\}$, 
\[
\lim_{r \to 0} \; \lim_{k \to \infty}E\left[ u_k, D_r(x_i) \right]:=m_i \geq \epsilon_{\mathrm{gap}}.
\]
\item \label{bct4} The energies satisfy the relation
\[
\lim_{k \to \infty} E[u_k] = E[u] + \sum_{i=1}^\ell m_i.
\]
\end{enumerate}
\end{lemma}
\begin{proof}
Items \eqref{bct1} and \eqref{bct3} follow by standard arguments using Proposition \ref{BTtools}.

 For \eqref{bct2}, choose $r'>0$ such that for all $x \in M$, $u(D_{r'}(x)) \subset \B_{\tau(X)/2}(u(x))$. Let $d_\Omega:= d_g(\partial \Omega, \{x_1,\dots, x_\ell\})$. There exists $K_{\Omega} \in \mathbb N$ such that for all $k \geq K_{\Omega}$ and, $x \in \Omega$, $0<r<d_\Omega/2$, $u_k(D_r(x)) \subset \B_{3\tau(X)/4}(u(x))$. Suppose to the contrary that the energy drops in the limit. Then there exists $y \in \Omega$ and $0<t<d_\Omega/2$ such that $\liminf_{k\to \infty}E\left[ u_k, D_t(x) \right]> E\left[ u,D_t(x) \right]$. But this contradicts the proof of compactness of minimizers in \cite[Lemma 2.3]{Banff1} and thus the claim holds.
 
 Finally,  \eqref{bct4} follows immediately from \eqref{bct2} and standard arguments.
 \end{proof}

Fix a constant 
\begin{equation}\label{eq:CR}
0<C_R \leq \min\left\{\frac \pi 3, \frac{\epsilon_{\mathrm{gap}}}2, C_{\mathrm{mon}}\frac{ \tau^2(X)}{16}\right\}.
\end{equation}
 Here $C_{\mathrm{mon}}$ is the monotonicity constant given in \cite[Theorem 3.4]{Banff2}. 
 
To understand the importance of the constants chosen in the following lemma, we provide a brief outline of their significance going forward. Let $r= \min\{d_g(x_i, x_j),i \neq j\}$, for $x_i$ from Lemma \ref{BT1}. In each ball $B_r(x_i)$, there are three regions of interest. In fact these regions will be on the pull back of this domain to $T_{x_i}M$ by the exponential map. We refer to this domain as $B_r(0) \subset T_xM$. 

In the next lemma, we choose these regions by choosing $\epsilon_{k,i}, \lambda_{k,i}, c_{k,i}$ where
\[
B_{k \lambda_{k,i}}(c_{k,i}) \subset B_{\epsilon_{k,i}}(c_{k,i}) \subset B_r(0).
\]By construction $\epsilon_{k,i}/(k\lambda_{k,i}) \to \infty$ so the annular scale is not uniform. The scales and how we choose them will help us to complete our main theorem. On the outer region $B_r(0) \backslash  B_{\epsilon_{k,i}}(c_{k,i})$, $u_k \to u$ uniformly in $C^0$, where $u$ is a harmonic map. On the inner region $B_{k \lambda_{k,i}}(c_{k,i})$, after an appropriate conformal transformation of the domain, the new maps $\overline u_{k,i}$ will converge to a ``bubble map" - a harmonic map from $\mathbb S^2$. The intermediate region $B_{\epsilon_{k,i}}(c_{k,i})\backslash B_{k \lambda_{k,i}}(c_{k,i})$ is called the ``neck region". The behavior of the sequence $u_k$ on this region is not captured by $u$ or by any of the bubble maps. Thus, the main objective is to determine whether any of the limiting information about $u_k$ escapes in the neck regions. Our main theorem demonstrates that no energy or image is lost in these necks.

\begin{lemma}\label{BT2}Consider a single bubble point $x_i$ with mass $m_i$. For simplicity we denote them in what follows as $x, m$. 
There exists a further subsequence, and constants $\epsilon_k \searrow 0$ and $C>0$ such that, identifying $u_k$ with $\exp_x^*u_k$ and $D_k:= D_{2\epsilon_k}(0) \subset T_xM$,   
\begin{enumerate}
\item \label{BT21}$E\left[ u_k, D_k \backslash D_{\epsilon_k/8k^2}(0) \right] \to 0$.
\item $u_k(\partial D_k) \subset \B_{C/k}(u(x))$. 
\item for $c_k:=(c_k^1,c_k^2)$ where 
\[
c_k^j= \frac {\int_{D_k} x^j |\nabla u_k|^2dx}{\int_{D_k}|\nabla u_k|^2 dx}, \quad \quad |c_k| \leq \epsilon_k/2k^2.
\]
\item \label{BT24} for
\[
\lambda_k := \min\{\lambda: \int_{D_{\epsilon_k}(0)\backslash D_{\lambda}(c_k)} |\nabla u_k|^2 dx \leq C_R\},\quad \quad \lambda_k \leq \epsilon_k/k^2.
\]

\end{enumerate}
\end{lemma}
Note that the proof of this lemma will not require $C^1$ convergence of $u_k$ to $u$, but instead uses the weaker convergence given by items \eqref{bct2}, \eqref{bct3} in Lemma \ref{BT1}.
\begin{proof}Let $\rho_0:=\frac 12 \min\{\mathrm{dist}(x_j, x): j \in \{1, \dots, \ell\}, x_j \neq x\}$. Choose $\epsilon_k \leq \min\{\frac 1k, \rho_0, \mathrm{inj}(M)\}$ to be the largest number such that
\[
\int_{D_k}|\nabla u|^2 \leq \frac m{16k^2}.
\]
We determine the subsequence inductively. For each $k \geq 1$, let $A_k:= D_k \backslash D_{\epsilon_k/8k^2}(0)$. With $\Omega:= \overline A_k$ fixed, items \eqref{bct2}, \eqref{bct3} of Lemma \ref{BT1} imply that there exists $N_k$ such that for all $n \geq N_k$, 
\begin{equation}\label{eq:annular}
\int_{A_k}|\nabla u_n|^2 \leq 2\int_{A_k}|\nabla u|^2< \frac m{8k^2}
\end{equation}and
\begin{equation}\label{eq:disk}
 \frac{8k^2-1}{8k^2}m\leq \int_{D_{\epsilon_k/8k^2}(0)}|\nabla u_n|^2 \leq \frac{8k^2+1}{8k^2}m.
\end{equation}Moreover, we may increase $N_k$ if necessary so that for all $n \geq N_k$, 
\begin{equation}\label{eq:boundarydist}
\sup_{y \in \partial D_k} d(u_n(y), u(y)) \leq \frac 1k.
\end{equation}Set $n_k = \max\{N_k , 1+n_{k-1}\}$. Then the first item follows from \eqref{eq:annular}. The existence of $C$ such that the second item holds follows from the Lipschitz regularity of $u$ combined with \eqref{eq:boundarydist}. The estimates on $c_k, \lambda_k$ follow from \eqref{eq:annular}, \eqref{eq:disk} and their definitions (cf. \cite[Section 6]{Parker}).
\end{proof}

We will need a conformal transformation of $D_k$ onto $S_k \subset \mathbb S^2$ such that $c_k \mapsto (0,0,1)$ and $\partial D_{\lambda_k}(c_k)$ maps to the equator. Let $\pi_{S^2}:  \mathbb S^2 \left(\subset \R^3\right) \to \R^2 \cong T_x M$ be a fixed stereographic projection that maps the equator to the unit circle and the north pole to the origin. Let $\Psi_k: \R^2 \to \R^2$ be given by
\[
	\Psi_k (x) := \lambda_k x + c_k.
\]
Then, the map $\Theta_k := \left( \Psi_k \circ \pi_{S^2} \right)^{-1} =  \pi_{S^2}^{-1} \circ \Psi_k^{-1}$ is a conformal transformation under which $c_k \mapsto (0,0,1)$ and $\partial D_{ \lambda_k}(c_k)$ maps to the equator. Define $S_k := \Theta_k \left( D_{2\epsilon_k} (0) \right)$. Now let
\begin{equation}\label{eq:bar-u}
\overline u_k:S_k \to X
\end{equation}
be defined as $\overline u_k \circ \Theta_k (x) = u_k (x)$ for all $x \in D_k$. Applying Proposition \ref{BTtools} to the maps $\overline u_k$ we obtain a result analogous to Lemma \ref{BT1} for maps from domains exhausting $\mathbb S^2$. For ease of notation, let $D_r^{\mathbb S^2}(y)$ denote a geodesic disk in $\mathbb S^2$ of radius $r$ and centered at $y \in \mathbb S^2$.

\begin{lemma}\label{BT3}Let $S^-_k$ represent the portion of $S_k$ in the southern hemisphere, $p^-$ denote the south pole, and $\epsilon_{\mathrm{gap}}$ be as in Proposition \ref{GapThm}. 
There exists a further subsequence $\{\overline u_k\}$ of harmonic maps with $\overline u_k:S_k \to X$, a harmonic map $\overline u:\mathbb S^2 \to X$, a collection of points $\{y_1, \dots, y_l\}\subset \mathbb S^2$ with corresponding masses $\{m_1, \dots, m_l\}$ such that
\begin{enumerate}
\item \label{BT31} $\overline u_k \to \overline u$ in $C^0$ uniformly on compact subsets of $\mathbb S^2 \backslash \{y_1, \dots, y_l, p^-\}$.
\item \label{BT33}  $\lim_{k \to \infty}E\left[ \overline u_k,S_k \right] = m$. 
\item \label{BT32} $\lim_{k \to \infty}E\left[ \overline u_k, S_k^- \right]=C_R$.
\item \label{BT34} for any open set $\Omega$ with $\overline{\Omega} \subset \mathbb S^2 \backslash \{y_1, \dots, y_l, p^-\}$, 
\[\lim_{k \to \infty}E\left[ \overline u_k, \Omega \right] = E\left[ \overline u, \Omega \right].
\]
\item \label{BT35} for all $r>0$ and $j \in \{1, \dots, l\}$,
\[
 \lim_{r \to 0} \; \lim_{k \to \infty} E\left[ \overline u_k, D_r^{\mathbb S^2}(y_j) \right] :=m_j \geq \epsilon_{\mathrm{gap}}.
 \]
 \item \label{BT35b} there exists $\tau(p^-) \geq 0$ such that
 \[
 \lim_{k \to \infty} E[\overline u_k,S_k] = E[\overline u,\mathbb S^2] + \tau(p^-)+ \sum_{i=1}^l m_i
 \]
 \item \label{BT35b2} the map $|\nabla \overline u|^2 + \tau(p^-)\delta_{p^-} + \sum_{i=1}^l m_i\delta_{y_i}$ has center of mass on the $z$-axis.
 \item \label{BT35c} if $E[\overline u,\mathbb S^2]<\epsilon_{\mathrm{gap}}$ then $E[\overline u,\mathbb S^2]=0$. In this case, $l>0$ and if $l=1$ then $\tau(p^-)=C_R$.
\item \label{BT36} $\overline u_k(\partial \Theta_k(D_{k\lambda_k}(c_k))) \subset \B_{C/k}(\overline u(p^-))$.
\item \label{BT37} 
$ E \left[  \overline{u}_k , \Theta_k \left( D_{2k\lambda_k}(c_k)  \backslash D_{k\lambda_k}(c_k) \right) \right]  \to 0$.
\end{enumerate}
\end{lemma}
\begin{remark}
Following the usual convention, if $E[\overline u,\mathbb S^2]=0$ then we say that $\overline u$ is a \emph{ghost bubble}.
\end{remark}
\begin{proof}Item \eqref{BT31} follows from arguments as in Lemma \ref{BT1} and item \eqref{BT33} follows from the choice of $D_k$ earlier. Observe also that
\begin{align*}
E\left[ \overline u_k, S_k^- \right] = E\left[ u_k, D_k \backslash D_{\lambda_k}(c_k) \right] 
=E\left[ u_k, D_k \backslash D_{\epsilon_k}(0) \right] + E\left[ u_k,  D_{\epsilon_k}(0)\backslash D_{\lambda_k}(c_k) \right].
\end{align*} Item \eqref{BT32} now follows from Lemma \ref{BT1} \eqref{bct2} and Lemma \ref{BT2} \eqref{BT21}, \eqref{BT24}. Items \eqref{BT34} -- \eqref{BT35b} follow the logic as in Lemma \ref{BT2}, though we must include $\tau(p^-)$ in \eqref{BT35b} since energy may concentrate at $p^-$. 
Item \eqref{BT35b2} holds since for $f \in C^\infty(\mathbb S^2, \mathbb R)$, by approximating by characteristic functions and appealing to the logic that gives \eqref{BT35b}, we conclude that
\[
\lim_{k \to \infty}\int_{\mathbb S^2 \cap S_k} f|\nabla \overline u_k|^2 d\mu_g = \int_{\mathbb S^2} f|\nabla \overline u|^2 d\mu_g + f(p^-) \tau(p^-)+ \sum_{i=1}^l f(y_i)m_i.
\]
The first part of item \eqref{BT35c} is immediate from the gap theorem. In that case, $l>0$ by items \eqref{BT33} and \eqref{BT32} and the fact that the $y_j$'s are in the northern hemisphere. When $l=1$, items \eqref{BT32} and \eqref{BT35b2} and the fact that $y_1$ must be in the northern hemisphere imply the result on $\tau(p^-)$. Item \eqref{BT36} follows as in Lemma \ref{BT2}. For item \eqref{BT37}, first notice that 
	\[
	E \left[  \overline{u} , \Theta_k \left( D_{2k\lambda_k}(c_k)  \right) \right]  \to 0 \quad as \quad k \to \infty.
	\]
	By item \eqref{BT34}, for each fixed $k\geq 1$ we can choose $N_k$ such that for all $n \ge N_k$,
	\[
\left|	E \left[  \overline{u}_n , \Theta_k \left( D_{2k\lambda_k} (c_k) \backslash D_{k\lambda_k}(c_k) \right) \right]  - E \left[  \overline{u} , \Theta_k \left( D_{2k\lambda_k}(c_k)  \backslash D_{k\lambda_k} (c_k)\right)  \right] \right|< \frac{1}{k}.
	\]
	Letting $n_k := \max\{n_{k-1}+1 , N_k\}$, we see that 
	\[
\lim_{k \to \infty}E \left[  \overline{u}_{n_k} , \Theta_k \left( D_{2k\lambda_k} (c_k) \backslash D_{k\lambda_k} (c_k)\right)  \right]  \to 0.
	\]
Renaming the sequence implies item \eqref{BT37}.
\end{proof}

\subsection{The bubble tree}\label{MapsSec}

Given a bubble point $x_i \in M$ from Lemma \ref{BT1}, by Lemma \ref{BT3} the maps $\overline u_{k,i}:S_{k,i} \to X$ converge to a map $\overline u_i:\mathbb S^2 \to X$ except at $\{y_{i1}, \dots, y_{il_i}, p^-\}$. The process then iterates at each $y_{ij}$ which allows us to obtain bubbles on bubbles. By Lemma \ref{BT3}, item \eqref{BT35c}, there can be at most $\min\{\Lambda/C_R, \log_2(\Lambda/\epsilon_{\mathrm{gap}})\}$ ghost bubbles. Since every non-ghost bubble contains at least $\epsilon_{\mathrm{gap}}$ energy, the process terminates. 

Prior to constructing the bubble tree, we prove two technical facts. First, given image curves $\gamma (\partial D_r)$ with small length and energy bounded by $C\delta/r$, there exists a coning off of $\gamma$ in $X$ which has energy bounded by $C\delta$. Second, the sequence of maps $u_k$ possess such curves on scale $\epsilon_k$ and $k\lambda_k$. These technical facts will be useful in the construction of our bubble tree.

\subsubsection{Coning off a curve}
We first demonstrate the existence of a coning off of a Lipschitz curve with energy depending on the energy of the curve.
\begin{definition}\label{def:coneoff}
Let $\gamma: \partial D_r \to \B_{\tau(X)}(P) \subset X$ be a Lipschitz map. We define the \emph{cone extension map} $\mathrm{Cone}(\gamma_{\partial D_r}): D_r \to X$ such that 
\[
\mathrm{Cone}(\gamma_{\partial D_r})(s,\theta) = \eta_\theta \left( \frac{s}{r} \right),
\] where $\eta_\theta:[0,1] \to X$ is the constant speed geodesic connecting $c_\gamma$ to $\gamma(\theta)$ and $c_\gamma$ is the circumcenter of $\gamma$.
\end{definition}

\begin{lemma}\label{lem:coneoff}
		Let $\left( D_{2r} , ds^2 + \mu^2(s,\theta)s^2 d\theta^2     \right)$ be a smooth disk such that $s^{-2}|1-\mu^2| + s^{-1} |\partial \mu^2| + |\partial^2 \mu^2| \leq c< \frac 14$ and $r<1$. Let $X$ be a compact locally CAT(1) space with injectivity radius $\tau(X)$. There exist $C, \delta >0$ depending on $X$ such that the following holds: for any  Lipschitz loop $\gamma: \partial D_r \to X$ with $E[\gamma, \partial D_r] < \frac{\delta}{r} $, the cone extension map $\mathrm{Cone}(\gamma_{\partial D_r}): D_r \to X$ exists and satisfies $E[\mathrm{Cone}(\gamma_{\partial D_r}),D_r] \le C\delta$. 
\end{lemma}

\begin{proof}By the Cauchy-Schwarz inequality,
\[
\mathrm{Length}^2(\gamma) \leq 2\pi \int_0^{2\pi} \left|\frac{\partial \gamma}{\partial \theta}\right|^2 d\theta \leq 2\pi \mu(r, \theta) \delta< \frac{\tau^2(X)}2
\]for sufficiently small $\delta$.  Thus $\gamma \subset \B_{C\delta^{1/2}}(c_\gamma)$ and $\mathrm{Cone}(\gamma_{\partial D_r})(D_r) \subset \B_{C\delta^{1/2}}(c_\gamma)$. For convenience, let $u:=\mathrm{Cone}(\gamma_{\partial D_r})$.

 By the $L-$convexity of CAT(1) spaces \cite[Definition 2.6 and Proposition 3.1]{ohta}, we deduce that 
\[
d\left(  \eta_{\theta_1}(t) , \eta_{\theta_2}(t)    \right) \le \left( 1 + C \delta    \right) t\, d \left( \gamma \left( \theta_1 \right) ,  \gamma\left( \theta_2 \right) \right).
\]

Now we estimate the directional derivative $\left| u_*\left( {\partial_\theta}   \right) \right|\left( s,\theta_0 \right) $. Let $\mu := \mu (s , \theta_0)$ and let $\|  \cdot \|$ denote the distance to $0$ with respect to the metric $g = ds^2 + \mu^2s^2d\theta^2$. Then by \cite[Section 1.9]{korevaar-schoen1}, for a.e. $(s, \theta_0) \in D_r$, 
\begin{align*}
\left| u_*\left( {\partial}_{ \theta}   \right) \right|\left( s,\theta_0 \right) &=  \lim_{h \to 0}\frac{d\left( u\left( s,\theta_0 \right)  ,  u\left( \mathrm{exp}_{(s,\theta_0)}\left(h\partial_\theta \right) \right) \right) }{|h|}  \notag \\  
&\le \lim_{h \to 0} \frac{d\left( u\left( s,\theta_0 \right)  ,  u\left( \|\mathrm{exp}_{(s,\theta_0)}\left({h\partial_\theta} \right) \| , \theta_0 \right) \right)}{|h|} \notag \\ 
&+ \lim_{h \to 0} \frac{d\left( u \left(\|\mathrm{exp}_{(s,\theta_0)}\left(h\partial_\theta\right) \|,\theta_0 \right)   ,  u \left(\|\mathrm{exp}_{(s,\theta_0)}\left(h{\partial_\theta} \right) \|, \mathrm{arg}_\theta\left(  \mathrm{exp}_{(s,\theta_0)}\left(h{\partial_\theta} \right)  \right)  \right) \right) }{|h|}. \notag 
\end{align*}
As the radial geodesics are constant speed, and using the $L-$convexity estimate, we deduce that
\begin{align*}
 \left| u_*\left({\partial_\theta}   \right) \right| \left( s,\theta_0 \right) & \le \frac{d\left(c_\gamma , \gamma\left(  \theta_0 \right) \right) }r \lim_{h \to 0} \frac{\| \mathrm{exp}_{(s,\theta_0)}\left(h{\partial_\theta}\right)  \| - s}{|h|}  \notag \\ &+ \lim_{h \to 0}\left(  1 + C \delta  \right) \frac{ \| \mathrm{exp}_{(s,\theta_0)}\left(h{\partial_\theta}\right)  \| }{r}\;   \frac{d \left(   \gamma\left( \theta_0 \right) ,    \gamma \left(   \mathrm{arg}_\theta\left(  \mathrm{exp}_{(s,\theta_0)}\left(h{\partial_\theta}\right)  \right)    \right)  \right)   }{ |h|}. \notag 
\end{align*}

Since ${\partial_s}$ and ${\partial_\theta}$ are perpendicular, the first variation formula implies that
\[
    \lim_{h \to 0} \frac{\| \mathrm{exp}_{(s,\theta_0)}\left(h{\partial_\theta} \right)  \| - s}{ |h|} = 0.
\]
Thus, 
\[
\| \mathrm{exp}_{(s,\theta_0)}\left(h{\partial_\theta}\right)  \| = s + o(|h|)|h|.
\]Let $\Delta \theta(h):=  \mathrm{arg}_\theta\left(  \mathrm{exp}_{(s,\theta_0)}\left(h{\partial_\theta}\right) \right)-\theta_0$. Then,
\begin{align*}
\lim_{h \to 0} \frac{d \left( \gamma\left( \theta_0 \right) ,    \gamma \left(   \theta_0+ \Delta \theta(h) \right)    \right)  }{h} &= \lim_{h \to 0} \frac{d \left(   \gamma\left( \theta_0 \right) ,    \gamma \left(   \theta_0+ \Delta \theta(h) \right)      \right) }{\Delta \theta(h)} \cdot \lim_{h \to 0} \frac{\Delta \theta(h)}{h}\\
&= \left| \frac{d \gamma}{d\theta}   \right|_{\theta = \theta_0} \cdot \mu.
\end{align*}
Therefore, the directional derivative $\left| u_*\left( \frac{1}{\mu s} \frac{\partial}{\partial \theta}   \right) \right|$ satisfies
\begin{equation*}
 \frac{1}{\mu s}\left|u_*(\partial_\theta)\right| = \left| u_*\left( \frac{1}{\mu s}{\partial_\theta}   \right) \right|\left( s,\theta_0 \right)  \le \frac{1 + C \delta}r\left| \frac{d \gamma}{d\theta}   \right|_{\theta = \theta_0} .
\end{equation*}
Moreover, one easily calculates, using the constant speed of $\eta$,  
\[
\left|  u_*(\partial_s)   \right| \left( s , \theta_0 \right) = \frac{1}{r} d\left( c_\gamma , \gamma\left(\theta_0  \right)   \right) \le \frac{C\delta^{1/2}}{r}. 
\]
It follows that
\begin{align*}
\left| \nabla u  \right|^2\left( s , \theta_0\right) 
&=\left| u_*({\partial_s} )\right|^2(s,\theta_0)+ \frac {1}{\mu^2 s^2}  \left| u_*\left( {\partial_\theta}   \right) \right|^2(s,\theta_0)\\
&\leq \frac {C^2\delta}{r^2} + \frac{( 1 + C \delta)^2}{r^2}\left| \frac{d \gamma}{d\theta}\right|^2(\theta_0).
\end{align*}
Increasing $C$ as necessary, 
\begin{align*}
E[u,D_r] &= \int_{D_r}  \left| \nabla u  \right|^2\left( s , \theta\right) \mu(s,\theta) s ds d\theta\\
& \leq \frac{C^2\delta}{r^2}\int_{D_r} \mu(s,\theta)\, s\,ds d\theta + \frac{\left(  1 + C\delta  \right)^2}{r^2}\int_{D_r} \left|  \frac{d \gamma}{d\theta} \right|^2  \mu(s,\theta)\, s\, ds\, d\theta  \\ 
&\le   2\pi C^2 \delta(1+cr^2) + (1+C\delta)^2 (1+cr^2)^2r E[\gamma, \partial D_r]\\
& \leq C \delta.
\end{align*}
\end{proof}

\subsubsection{Curves with small length}
In subsection Section \ref{NecSec}, we prove the no-neck property and energy quantization using the isoperimetric inequality which applies to conformal harmonic maps. Thus, we want to find scales with small image length for a conformal suspension of $u_k$ which in turn implies small image length for the original maps $u_k$. 
We begin by recalling a modification of the previous conformalization scheme (see \cite[Theorem 2.3.4]{jost}). 

\begin{lemma}\label{lem:conformal-suspension-2}
Let $u: D_1 \to X$ be a harmonic map with $E\left[ u , D_1 \right] \le \Lambda$. Then, there exists a conformal harmonic map $\tilde{u}: D_1 \to X \times \mathbb{C}$ of $u$ and a universal constant $c_1>0$ such that for all $x \in D_{1/2}$, 
\begin{equation*}
	\left|   \nabla \tilde{u}  (x)\right|^2 \le \left|   \nabla u (x)\right|^2+ 1 + c_1^2 \Lambda^2.
\end{equation*}
\end{lemma}

\begin{proof}
We construct $\nu: D_1 \to \C$ to satisfy
	\begin{equation*}
		\begin{cases}
			\partial_{{z}} \overline \nu = 1 & \quad in \quad D_1,\\
			\partial_z \nu = -\frac{1}{4}\Phi_u  & \quad in \quad D_1, \\
			\Delta \nu = 0 & \quad in \quad D_1.
		\end{cases}
	\end{equation*}
To do this, let $\Psi$ be a holomorphic function with $\partial_z \Psi = -\frac{1}{4} \Phi_u$ where $\Phi_u$ is the Hopf function. Since $\Phi_u$ is holomorphic, $\nu(z) := \overline{z} + \Psi(z) $ satisfies the above conditions. Moreover, $\Phi_\nu = 4 \partial_z \nu \partial_z \overline{\nu} = -  \Phi_u$. Let $\tilde u:D_1 \to X \times \mathbb C$ such that $\tilde u(x):= (u(x), \nu(x))$. By construction, 
\begin{equation*}
	\Phi_{\tilde{u}} = \Phi_u + \Phi_\nu = \Phi_u - \Phi_u = 0
\end{equation*}and thus $\tilde u$ is conformal. 
Since $\Phi_u \in L^1\left( D_1 \right)$ is holomorphic on $D_1$, using the Cauchy integral formula there exists $c_1>0$ such that for all $x \in D_{1/2}$,  
\begin{equation*}
	\left|  \Phi_u(x)  \right| \le 4c_1 \Lambda.
\end{equation*}
Therefore, for all $x \in D_{1/2}$, 
\begin{equation*}
		\left|   \nabla \tilde{u}  \right|^2 = \left|   \nabla u \right|^2 + \left|  \nabla \nu    \right|^2 = \left|   \nabla u \right|^2 + 1 + \frac{\left|  \Phi_u  \right|^2}{16} \le \left|   \nabla u \right|^2 + 1 + {c_1^2}\Lambda^2.
\end{equation*}

\end{proof}
\begin{definition}
Henceforth we refer to the above constructed $\tilde u$ as the \emph{conformal suspension} of $u$.
\end{definition}

\begin{lemma}\label{lem:susp-area}
	Let $u$ and $\tilde{u}$ be as in Lemma~\ref{lem:conformal-suspension-2}. For any $ \Omega \subset D_{1/2}$ and any $0 < r <\frac 12$, 
		\begin{equation*}
		2 \mathrm{Area} \left[ \tilde{u} \left(\Omega \right) \right] = E\left[  \tilde{u} , \Omega   \right] \le E\left[  u ,\Omega \right] + \mathrm{Area}\left[\Omega\right] \left(  1 + {c_1^2}\Lambda^2  \right)
	\end{equation*}
and
\begin{equation*}
		\mathrm{length} \left[ \tilde{u} (\partial D_{r}) \right] \le \mathrm{length} \left[u (\partial D_{r}) \right] + \mathrm{length}[\partial D_r] \left(   1 + {c_1}\Lambda \right).
\end{equation*}Note that length and area of domain regions are taken with respect to the metric $g$. 
\end{lemma}

\begin{proof}
	Since $\tilde{u}$ is conformal, twice its area coincides with the total energy therefore,
	\begin{align*}
		2 \mathrm{Area} \left[ \tilde{u} \left(\Omega  \right]   \right) &= E\left[  \tilde{u} , \Omega \right] = \int_{\Omega} \; \left|   \nabla \tilde{u}  \right|^2 d\mu_g\\  &\le \int_{\Omega} \;\left( \left|   \nabla u \right|^2+ 1 + {c_1^2}\Lambda^2 \right) d\mu_g\\ &= E\left[  u ,\Omega  \right] +  \mathrm{Area}\left[\Omega  \right] \left(  1 + {c_1^2}\Lambda^2  \right). 
	\end{align*}
Similarly, letting $d\sigma_r$ denote the length measure on $\partial D_r$, 
\begin{align*}
 \mathrm{length} \left[ \tilde{u} (\partial D_{r}) \right] &= \int_{\partial D_{r}} \; \left|  \partial_\theta \tilde{u}   \right| \; d\sigma_r = \int_{\partial D_{r}} \;  \left|  \partial_\theta u  \right| + \left|  \partial_\theta \nu   \right| \; d\sigma_r \\  &\le \mathrm{length} \left[ u (\partial D_{r}) \right]  + \int_{\partial D_r} \; 1 + {c_1} \Lambda  \; d\sigma_r \\  &= \mathrm{length} \left[ u (\partial D_{r}) \right] + \mathrm{length}[\partial D_r]\left(   1 + {c_1}\Lambda \right). 
\end{align*}
\end{proof}

Let $x$ be a fixed bubble point and choose $\rho_0<\frac 12$ so that $D_{2 \rho_0}(x)$ does not contain any other bubble points. Let $\tilde u_k$ denote the conformal suspension of each $u_k|_{D_{2\rho_0}(x)}$ as in Lemma \ref{lem:conformal-suspension-2}. Recall that $\epsilon_k, \lambda_k$ are chosen in Lemma \ref{BT2} and an outline of their significance is given in the paragraph preceding that lemma. The next lemma provides precise scales, comparable to $\epsilon_k$ and $k \lambda_k$, on which we can apply our cone extension lemma. These two scales will determine the boundary of the neck region.
\begin{lemma}\label{lem:length}There exist sequences $r_k \in [\epsilon_k/4, \epsilon_k/2]$ and  $s_k \in [k\lambda_k, 2 k \lambda_k]$ such that 
\[
\lim_{k \to \infty}{r_k}E[\tilde u_k, \partial D_{r_k}(c_k)] = 0,
\]
\[
\lim_{k \to \infty}s_k E[\tilde u_k, \partial D_{s_k}(c_k)] = 0.
\]As a consequence,
\[
\lim_{k \to \infty}\mathrm{length}[\tilde u_k(\partial D_{r_k}(c_k))]  =0,
\]
\[
\lim_{k \to \infty}\mathrm{length}[\tilde u_k(\partial D_{s_k}(c_k))]  =0.
\]
\end{lemma}
\begin{proof}
As $\epsilon_k \to 0$, for each map $\tilde u_k$ we consider the metric in the tangent space $(D_{2\epsilon_k}(0), ds^2 + \mu^2_k(s,\theta)s^2 d\theta^2)$ where $s^{-2}|1- \mu_k^2|+s^{-1} |\partial \mu_k^2| + |\partial^2 \mu_k^2| \leq \alpha_k$ where $\alpha_k \to 0$. Let $r_k \in [\epsilon_k/4, \epsilon_k/2]$ such that $E[\tilde u_k, \partial D_{r_k}(c_k)] = \min_{r \in [\epsilon_k/4, \epsilon_k/2]}E[\tilde u_k, \partial D_{r}(c_k)] $. Then
\begin{align*}
\frac {\epsilon_k}4E[\tilde u_k, \partial D_{r_k}(c_k)]   &\leq \int_{\epsilon_k/4}^{\epsilon_k/2}\int_0^{2\pi} \frac 1{s \mu_k}\left|\frac{\partial \tilde u_k}{\partial \theta}\right|^2 d\theta ds   \\
&\leq E\left[ \tilde u_k, D_{\epsilon_k/2}(c_k) \backslash D_{\epsilon_k/4}(c_k) \right]\\
& \leq E\left[ u_k, D_k \backslash D_{\epsilon_k/8k^2}(0) \right] + \mathrm{Area}\left[ D_k \right] \left(1+ {c_1^2}\Lambda^2 \right) 
\end{align*}where the last inequality follows from Lemma \ref{lem:susp-area} and the fact that $ D_{\epsilon_k/2}(c_k) \backslash D_{\epsilon_k/4}(c_k)\subset D_k \backslash D_{\epsilon_k/8k^2}(0)$. By Item \eqref{BT21} of Lemma \ref{BT2} and the fact that $\mathrm{Area}\left[ D_k \right] \leq c \epsilon_k^2$, the final expression tends to zero in $k$. Since $r_k/2 \leq {\epsilon_k}/4$, the desired result follows.

To find the $s_k$'s we use item \eqref{BT37} of Lemma~\ref{BT3} in place of item \eqref{BT21} of Lemma \ref{BT2} and follow a similar reasoning as above.

The length estimates then follow immediately from Cauchy-Schwarz.
\end{proof}

\subsubsection{The base, neck, and bubble maps}

Around each bubble point $x_i \in \{x_1, \dots, x_\ell\}$, there are three domains of interest. In $D_k(x_i) = D_{2\epsilon_{k,i}}(0)\subset T_{x_i}M$, we consider the disks $D_{r_{k,i}}(c_{k,i}), D_{s_{k,i}}(c_{k,i})$ and the annulus between them 
\[
A_{k,i}':=D_{r_{k,i}}(c_{k,i})\backslash D_{s_{k,i}}(c_{k,i}).
\] 
Here $\epsilon_{k,i}, \lambda_{k,i}, c_{k,i}$ are given by Lemma \ref{BT2} and $r_{k,i}, s_{k,i}$ are given by Lemma \ref{lem:length}.

We define the \emph{neck maps} $u_{k,i}|_{ A_{k,i}'}:A_{k,i}' \to X$. To define the \emph{extended base maps}, let
\[
\underline u_k(x):= \left\{ \begin{array}{ll} u_k(x) & \text{if } x \in M \backslash \cup_{i=1}^\ell \mathrm{exp}(D_{r_{k,i}}(c_{k,i}))\\
\mathrm{Cone}(u_k|_{\partial D_{r_{k,i}}(c_{k,i})}) & \text{if } x \in \mathrm{exp}(D_{r_{k,i}}(c_{k,i})).
\end{array}\right.
\]By Lemmas \ref{BT1}, \ref{lem:length}, and \ref{lem:coneoff}, $\underline u_k \to u$ in $C^0$ uniformly on $M$ and
\[
\lim_{k \to \infty} E[\underline u_k, M] = E[u, M].
\]

Similarly, the \emph{extended bubble maps} will cone off the maps $\overline u_{k,i}$. Let $\underline{\overline u}_{k,i}:\mathbb S^2 \to X$ such that
\[
\underline{\overline u}_{k,i}(x):= \left\{ \begin{array}{ll}
\overline u_{k,i}(x) & \text{if } x \in \Theta_{k,i} (D_{s_{k,i}}(c_{k,i}))\\
\mathrm{Cone}(\overline u_{k,i}|_{\Theta_{k,i} ( \partial D_{s_{k,i}}(c_{k,i}))}(x) &\text{otherwise}.
\end{array}\right.
\]
By Lemmas \ref{BT3}, \ref{lem:length}, and \ref{lem:coneoff}, for each $i \in \{1 , \dots, \ell\}$, $\underline{\overline u}_{k,i} \to {\overline u}_{i}$ uniformly in $C^0$ on $\mathbb S^2 \backslash \{y_{i1}, \dots y_{il_i}\}$ and 
\[
\lim_{k \to \infty} E[\underline{\overline u}_{k,i} ,\mathbb S^2] = E[ {\overline u}_i,\mathbb S^2] + \sum_{j=1}^{l_i} m_{ij}.
\]
Note that the term $\tau_i(p^-)$ is lacking from the above limit and the uniform convergence happens across $p^-$. This occurs since we have removed the neck map portion from the extended bubble maps and replaced it by the coning off which has energy and diameter converging to zero. Moreover,  the energy contained in the neck maps is exactly
\[
\tau_i(p^-) = \lim_{k \to \infty} E[u_{k,i}, A_{k,i}'].
\] We will show in the next subsection that $\tau_i(p^-)=0$ and $\mathrm{diam}(u_{k,i}(A_{k,i}')) \to 0$ and thus the $C^0$ limit and the limit of the energies of the extended bubble maps are the same as the limit for the original maps. 

As mentioned previously, the extended bubble map process now iterates and we construct maps $\underline {\overline u}_{k,ij}:\mathbb S^2 \to X$ where $j \in \{1, \dots, l_i\}$. These maps converge, away from finitely many points $y_{ijm}$, $m \in \{1, \dots, l_{ij}\}$, to some map $ {\overline u}_{ij}:\mathbb S^2 \to X$ with an analogous energy limit to what we saw above. 

\subsubsection{Constructing the bubble tree}\label{BTconst}
We now construct the bubble tree and the bubble domain. The bubble tree consists of vertices and edges where each vertex represents a harmonic map and each edge represents a bubble point. The base vertex of the tree is the map $u:M \to X$ and the $\ell$ edges emanating from the base vertex are the points $x_i$. The edges $x_i$ connect to the vertices $ {\overline u}_{i}:\mathbb S^2 \to X$ and the edges emanating from each of these vertices are the bubble points $y_{ij}$, $j \in \{1, \dots, l_i\}$. The tree is finite as the process terminates. 

The bubble tower is the disjoint union of $M$ and a collection of $\mathbb S^2$'s, where each $\mathbb S^2$ is associated with a vertex in the bubble tree. Indeed, following \cite{Parker}, we may consider a \emph{bubble tower} $T$ in the following manner. Let $SM$ be an $\mathbb S^2$ bundle over $M$. Compactifying the vertical tangent space of $S M \to M$ yields an $\mathbb S^2$ bundle $S^2M$ over $SM$. Iterating this process then yields a tower of $\mathbb S^2$ fibrations. For clarity, the point $y_{i_1i_2\dots i_n}$ lies in $S^{n-1}M$.

A \emph{bubble domain} at level $n$ is a fiber $F$ of $ S^nM \to  S^{n-1}M$ and a \emph{bubble tower} is a finite union of bubble domains such that the projection of $T \cap  S^nM$ lies in $T \cap  S^{n-1}M$. 

Given the sequence $u_k:M \to X$, applying Lemmas \ref{BT1}, \ref{BT2}, \ref{BT3} determines a unique bubble tower $T= M \bigcup \left(\cup_I \mathbb S^2_I\right)$ where $I$ is indexed over all of the bubble points in the process. We define a sequence of bubble tower maps $\overline{\underline u}_{k,I}:T \to X$ such that ${\underline u}_k:M \to X$ and $\overline{\underline u}_{k,I}:\mathbb S^2_I \to X$. Letting $u, \overline {\underline u}_I$ index the limit maps, observe that
\begin{equation}
\lim_{k \to \infty}E[\overline{\underline u}_{k,I},T] = E[\overline{\underline u}_I,T]
\end{equation}and $\overline{\underline u}_{k,I} \to \overline{\underline u}_I$ in $C^0$ uniformly on $T$.

\subsection{Energy quantization and the no neck property}\label{NecSec}
In this subsection, we study the neck maps and use the isoperimetric inequality to prove that the energy of neck maps vanish in the limit. Then by monotonicity, the diameter of the neck maps must also vanish. Taken with the previous subsections, these results immediately imply Theorem \ref{MAIN}.

Consider a single neck map $u_k: A_k' \to X$ where $A_k':=D_{r_k}(c_k) \backslash D_{s_k}(c_k)$. 

\begin{lemma}[Vanishing Neck Energy and Length] The following holds:
	\[
	\limsup_{k \to \infty} E\left[ u_k, A_k' \right] =0,
	\]
	\[
	\limsup_{k \to \infty}\mathrm{diam}\left[u_k(A_k')\right] =0.
	\]
\end{lemma}

\begin{proof} Let $\tilde u_k$ denote the conformal suspension of each $u_k|_{D_{r}(x)}$ as in Lemma \ref{lem:conformal-suspension-2}. 
By \eqref{eq:CR} and Lemma \ref{lem:susp-area}, for all sufficiently large $k$,  $\mathrm{Area}\left[\tilde{u}_k(A_k')  \right]< \frac \pi 3$.

By Lemma \ref{lem:length}, for any $0<\delta\leq \tau(X)/4$ there exists a $K$ such that for all $k \geq K$ there exist points $P_k, Q_k \in X \times \mathbb C$ such that $\tilde u_k(\partial A_k') \subset \B_\delta(P_k) \cup \B_\delta(Q_k)$. Now suppose that there exists $R_k \in \tilde u_k(A_k')$ such that $R_k \notin 
\B_{2\delta}(P_k) \cup \B_{2\delta}(Q_k)$. Then, applying the monotonicity formula of \cite[Theorem 3.4]{Banff2} to $\tilde u_k(A_k') \cap \B_\delta(R_k)$, 
\[
 C_{\mathrm{mon}} \delta^2 \leq \mathrm{Area}\left[\tilde u_k(A_k')\right] .
\]On the other hand, by Lemma \ref{lem:susp-area}, using the fact that $A_k' \subset D_{\epsilon_k}(0)\backslash D_{\lambda_k}(c_k)$, and recalling the definition of $C_R$ from \eqref{eq:CR}, 
\[
2\mathrm{Area}\left[\tilde u_k(A_k')\right] \leq E[u_k, D_{\epsilon_k}(0)\backslash D_{\lambda_k}(c_k)]+ \mathrm{Area}[A_k'](1+c_1^2\Lambda^2)\leq C_{\mathrm{mon}}\frac{\tau^2(X)}{16}+ \frac {C}{k^2}(1+c_1^2\Lambda^2).\]
This implies a contradiction for $\delta = \tau(X)/4$ and $k$ sufficiently large. It follows that for $k$ large enough, $\tilde u_k( A_k')  \subset \B_{\tau(X)}(P_k)$.  Thus each $\tilde u_k:A_k' \to X$ satisfies the hypotheses of the isoperimetric inequality, Theorem \ref{thm:isoperimetric}. By Lemma \ref{lem:length}, it thus follows that
\[
E\left[ u_k, A_k' \right] \leq E\left[ \tilde u_k,A_k' \right] = 2\mathrm{Area}\left[\tilde u_k(A_k')\right] \leq \frac{27\pi}2 \mathrm{length}^2\left[\tilde u_k(\partial A_k')\right]\to 0.
\]

With this improvement on the area estimate, for any fixed $\delta>0$ we may choose $N$ large enough so that for all $k \geq N$,
$\mathrm{Area}\left[\tilde u_k(A_k')\right] < C_{\mathrm{mon}}\delta^2/2$ and there exist points $P_k, Q_k \in X \times \mathbb C$ such that $\tilde u_k(\partial A_k') \subset \B_\delta(P_k) \cup \B_\delta(Q_k)$. If there exists $R_k \in \tilde u_k(A_k')$ such that $R_k \notin 
\B_{2\delta}(P_k) \cup \B_{2\delta}(Q_k)$ then by the same argument as above, the monotonicity formula implies a contradiction. Therefore, $\tilde u_k(A_k') \subset \B_{4\delta}(P_k)$. It follows that
\[
\lim_{k \to \infty} \mathrm{diam}\left[u_k(A_k')\right] \leq \lim_{k \to \infty} \mathrm{diam}\left[\tilde u_k(A_k')\right]=0.
\]
\end{proof}


\end{document}